\documentclass[12p, reqno]{amsart}
\usepackage{amsfonts}
\usepackage{amsmath,amsthm,graphicx,mathrsfs,url   }
\usepackage{amssymb}
\usepackage{enumitem}

\usepackage{graphicx}
\usepackage[usenames,dvipsnames]{color}
\usepackage[colorlinks=true, linkcolor=Red, citecolor=Green]{hyperref}
\newtheorem{thm}{Theorem}[section]

\newtheorem{prop}{Proposition}[section]

\numberwithin{equation}{section}

\begin{document}

\def\re{{\rm Re}\:}
\def\im{{\rm Im}\:}
\def\R{{\mathbb R}}
\def\C{{\mathbb C}}
\def\N{{\mathbb N}}
\def\DD{{\mathbb D}}
\def\S{{\mathbb S}}
\def\rr{{\cal R}}
\def\e{\emptyset}
\def\dQ{\partial Q}
\def\dk{\partial K}
\def\endofproof{{\rule{6pt}{6pt}}}
\def\di{\displaystyle}
\def\dist{\mbox{\rm dist}}
\def\u-{\overline{u}}
\def\du{\frac{\partial}{\partial u}}
\def\dv{\frac{\partial}{\partial v}}
\def\dt{\frac{d}{d t}}
\def\dx{\frac{\partial}{\partial x}}
\def\con{\mbox{\rm const }}
\def\Box{\spadesuit}
\def\ii{{\bf i}}
\def\curl{{\rm curl}\,}
\def\dive{{\rm div}\,}
\def\grad{{\rm grad}\,}
\def\dist{\mbox{\rm dist}}
\def\pr{\mbox{\rm pr}}
\def\pp{{\cal P}}
\def\supp{\mbox{\rm supp}}
\def\Arg{\mbox{\rm Arg}}
\def\In{\mbox{\rm Int}}
\def\Re{\mbox{\rm Re}\:}
\def\li{\mbox{\rm li}}
\def\ep{\epsilon}
\def\tr{\tilde{R}}
\def\be{\begin{equation}}
\def\ee{\end{equation}}
\def\cn{{\mathcal N}}
\def\sn{{\mathbb  S}^{n-1}}
\def\Ker {{\rm Ker}\:}
\def\el{E_{\lambda}}
\def\Rc{{\mathcal R}}
\def\la{\langle}
\def\ra{\rangle}
\def\Ko{\Ker G_0}
\def\Kd{\Ker G_b}
\def\hc{{\mathcal H}}
\def\caH{\hc}
\def\caO{{\mathcal O}}
\def\la{\langle}
\def\ra{\rangle}
\def\sp{\sigma_{+}}
\def\pa{\partial}
\def\et{|\eta'|^2}
\def\lg{L^2(\Gamma)}
\def\h1{H^1_h(\Gamma)}
\def\fr{\frac}
\def\il{-\ii \lambda}
\def\2l{-\fr{1}{2\lambda}}
\def\dd{\fr{d}{d\lambda}}
\def\oc{{\mathcal O}}
\def\nc{{\mathcal N}}
\def\cc{{\mathcal C}}
\def\bc{{\mathcal B}}
\def\tr{{\rm tr}}
\def\th{\tilde{h}} 
\def\pc{{\mathcal P}}

\title[dissipative eigenvalues] {Asymptotic of the dissipative eigenvalues of Maxwell's equations}

\author[V. Petkov]{Vesselin Petkov}
\address{Institut de Math\'ematiques de Bordeaux, 351,
Cours de la Lib\'eration, 33405  Talence, France}
\email{petkov@math.u-bordeaux.fr}

\begin{abstract}
 Let $\Omega = \R^3 \setminus \bar{K}$, where $K$  is  an open bounded domain with smooth boundary $\Gamma$. Let  $V(t) = e^{tG_b},\: t \geq 0,$ be the semigroup
  related to  Maxwell's equations  in $\Omega$  with dissipative boundary condition $\nu \wedge (\nu \wedge E)+ \gamma(x) (\nu \wedge H) = 0, \gamma(x) > 0, \forall x \in \Gamma.$  We study the case when $\gamma(x) \neq 1, \: \forall x \in \Gamma,$  and we establish a Weyl formula for the counting function of the  eigenvalues of $G_b$ in a polynomial neighbourhood of the negative real axis.

\end{abstract}

\maketitle
\section{Introduction}

Let $K \subset \{ x\in \R^3: \: |x| \leq a\}$ be an open connected domain 
and let $\Omega = \R^3 \setminus \bar{K}$  
be connected  domain with $C^{\infty}$ smooth boundary $\Gamma$.
Consider the boundary problem
\begin{equation}  \label{eq:1.1}
\begin{cases} 
\partial_t E = \curl H,\qquad \partial_t H = -\curl E \quad {\rm in}\quad \R_t^+ \times \Omega,
\\
\nu \wedge (\nu \wedge E )+\gamma(x)(\nu \wedge H) = 0 \quad{\rm on} \quad \R_t^+ \times \Gamma,
\\
E(0, x) = E_0(x), \qquad H(0, x) = H_0(x)
\end{cases}
\end{equation}
with initial data $F_0 = (E_0, H_0) \in  \hc = L^2(\Omega; \C^3) \times L^2(\Omega; \C^3).$
Here $\nu(x)$ is the unit outward normal at $x \in \Gamma$ pointing into $\Omega$, 
and $\gamma(x) \in C^{\infty}(\Gamma)$ satisfies $\gamma(x) > 0$ for all $ x \in \Gamma.$ The solution of the problem (\ref{eq:1.1}) is described  by a contraction semigroup 
$$(E, H)(t)  = V(t)F_0 = e^{tG_b} F_0,\: t \geq 0,$$
 where the generator $G_b$ is the operator
 $$G = \begin{pmatrix} 0 & \curl \\
 - \curl &  0\end{pmatrix}$$
  with  domain $D(G_b) \subset \mathcal H$
which is the closure in the graph norm $\| |u| \| = ( \|u\|^2 + \|G u\|^2)^{ 1/2}$ 
of functions $u = (v, w) \in C_{(0)}^{\infty} (\R^3; \C^3) \times C_{(0)}^{\infty} (\R^3; \C^3)$ satisfying the boundary condition $\nu \wedge (\nu \wedge v) + \gamma (\nu \wedge w) = 0$ on $\Gamma.$ \\

In \cite{CPR1} it was proved that the spectrum of $G_b$ in the open half plan $\{ z \in \C:\: \re z < 0\}$
is formed by isolated eigenvalues with finite multiplicities. Notice that if $G_b f = \lambda f$ with $\re \lambda < 0$, the solution $u(t, x) = V(t) f = e^{\lambda t} f(x) $ of (\ref{eq:1.1}) has exponentially decreasing global energy. Such solutions are called {\bf asymptotically disappearing} and they are important for the scattering problems (see \cite{CPR1}, \cite{CPR2}, \cite{P1}, \cite{P2}). In particular, the eigenvalues $\lambda$ with $\re \lambda \to -\infty$ imply  a very fast decay of the corresponding solutions. Let $\sigma_p(G_b)$ be the point spectrum of $G_b$. Concerning the scattering problems, we mention three properties related to the existence of eigenvalues of $G_b$. First, let $W_{\pm}$ be the wave operators
$$W_{-} f = \lim_{t \to + \infty} V(t) J U_0(-t) f,\:\: W_{+} f = \lim_{t \to + \infty} V^*(t) J U_0(t) f,$$
where $U_0(t)$ is the unitary group in $\hc_0 = L^2(\R^3; \C^3) \times L^2(\R^3; \C^3)$ related to the Cauchy problem for Maxwell system,  $J: \hc_0 \rightarrow \hc$ is a projection and $V^*(t) = e^{t G^*} $ is the adjoint semigroup (see \cite{LP}, \cite{P}). 
If $\sigma_p(G_p) \cap \{z: \Re z < 0\} \neq \emptyset,$ the wave operators $W_{\pm}$ are not  complete (see \cite{CPR1}), that is ${\rm Ran}\: W_{+}  \neq  {\rm Ran}\: W_{-}$ and we cannot define the scattering operator by $S = W_{+}^{-1} \circ W_{-}$. We may define the scattering operator by using another evolution operator (see \cite{LP}, \cite{P}). Second, in a suitable representation the scattering operator becomes an operator valued function $S(z): L^2(\S^2; \C^3) \rightarrow L^2(\S^2;\C^3), \: z \in \C,$ and Lax and Phillips (see \cite{LP}) proved that the existence of $z_0, \: \im z_0> 0,$ for which the kernel of $S(z_0)$ is not trivial implies  $\ii z_0 \in \sigma_p(G_b)$. The existence of such $z_0$ leads to problems in inverse scattering. Third, for dissipative systems Lax and Phillips developed a scattering theory in \cite{LP} and they introduced the representation of the energy space $\mathcal H$ as a direct sum $\hc = D_{a}^{-} \oplus K_{a} \oplus D_{a}^+$. A function $f$  is called {\it outgoing} (resp. {\it incoming}) if its component in $D_{a}^{-}$ (resp. $D_{a}^{+}$) is vanishing.  If $f$ is an eigenfunction with eigenvalue $\lambda \in \sigma_p(G_b),\: \Re \lambda < 0$, it is easy to see that  $f$ is incoming and, moreover, $V(t)f$ remains incoming for all $t \geq 0.$  On the other hand, $V^*(t) f$ is not converging to 0 as $t \to + \infty.$ In fact, assuming $V^*(t)f \to 0$ for $t \to +\infty$, by the result in \cite{CPR1} one deduces that $f$ must be disappearing, that is there exists $T > 0$ such that $V(t) f = 0$ for $t \geq T$ which is impossible for an eigenfunction.

The existence of infinite number eigenvalues of $G_b$ presents an interest for applications. However to our best knowledge this problem has been studied only for the ball $B_3 = \{x\in \R^3,\:|x| < 1\}$ assuming $\gamma$ constant (see \cite{CPR2}). It was proved in \cite{CPR2} that for $\gamma = 1$ there are no eigenvalues in $\{z \in \C: \Re z < 0\}$, while for $\gamma \neq 1$ there is always an infinite number of negative real eigenvalues $\lambda_j$ and with exception of one they satisfy the estimate 
\begin{equation} \label{eq:1.2}
\lambda_j  
\ \leq\
 - \frac{1}{\max \{ (\gamma_0 - 1), \sqrt{\gamma_0 - 1}\}} = -c_0 \,,
\end{equation}
where $\gamma_0 = \max\{\gamma, \frac{1}{\gamma}\}.$ On the other hand, a Weyl formula for the counting function of the negative eigenvalues of $G_b$ for $K = B_3$ and $\gamma \neq 1$ constant has been established in \cite{CP}.\\

 The distribution of the eigenvalues of $G_b$ in the complex plane has been studied in \cite{CPR2} and it was established that if $\gamma(x) \neq 1, \forall x \in \Gamma$, then for every $\epsilon > 0$ and every $M \in \N$, the eigenvalues lie in $\Lambda_{ \ep} \cup {\mathcal R}_{M}$, where 
$$\Lambda_{\ep} = \{z \in \C: \: |\Re z | \leq C_{\ep} (1 + |\im z|^{1/2 + \ep} ), \: \Re z < 0\},$$
$${\mathcal R}_M = \{z \in \C:\: |\im z | \leq C_M (1 + |\Re z|)^{-M} ,\: \Re z < 0\}.$$
 
An eigenvalue $\lambda_j \in \sigma_p(G_b) \cap \{ z: \: \Re z < 0\}$ has (algebraic) multiplicity given by
$${\rm mult}(\lambda_j)= {\rm tr}\: \frac{1}{2 \pi \ii} \int_{|\lambda_j - z|= \epsilon} (z - G_b)^{-1} dz,$$
where $0 < \epsilon \ll 1.$  Introduce the set
$$\Lambda := \{z \in \C:\: |\im z | \leq C_2 (1 + |\Re z|)^{-2} ,\: \Re z \leq -C_0 \leq -1\}.$$
We choose $C_0 \geq 2C_2$ and ${\mathcal R}_M \subset \Lambda,\: M \geq 2$ modulo a compact set containing a finite number eigenvalues.

Throughout this paper we assume that either $0 < \gamma(x) < 1, \forall x \in \Gamma$ or $1 < \gamma(x), \forall x \in \Gamma.$ Our purpose  is to prove the following
\begin{thm}
Let $\gamma(x)  \neq 1,\: \forall x \in \Gamma,$ and let $\gamma_0(x)  = \max \{\gamma(x) , \frac{1}{\gamma(x)}\}.$ Then the counting function of the eigenvalues in $\Lambda$ counted with their multiplicities for $r \to \infty$ has the asymptotic
\begin{equation} \label{eq:1.3} 
\sharp \{ \lambda_j \in \sigma_p(G_b) \cap \Lambda:\: |\lambda_j | \leq r, \: r \geq C_{\gamma_0}\} =  \frac{1}{4\pi}\Bigl( \int_{\Gamma} (\gamma_0^2(x) - 1)dS_x\Bigr) r^{2} + {\mathcal O}_{\gamma_0} (r).
\end{equation} 
\end{thm}
The proof of the above theorem follows the approach in \cite{SjV} and \cite{P2}. In comparison with \cite{P2}, we will discuss briefly some difficulties and new points. For the analysis of $\sigma_p(G_b)$ we prove in Section 2 a trace formula involving the operator $C(\lambda) f= \nc(\lambda) f + \frac{1}{\gamma(x)} (\nu \wedge f),$ where $\nc(\lambda) f = \nu \wedge H\vert_{\Gamma}$ and $(E, H)$ is the solution of the problem (\ref{eq:2.2}) with $U_1= U_2 = 0.$  Setting $\lambda = -\frac{1}{\th},\: \th = h( 1 + \ii t)$ with $0 < h \leq h_0,\: t \in \R,\:\: |t| \leq h^2,$ we are going to study the semiclassical problem (\ref{eq:2.8}) with $z = -\ii(1 + \ii t)^{-1}$. In a recent work Vodev \cite{V1} constructed a semiclassical parametrix for this problem assuming $\theta = |\im z| \geq h^{2/5- \epsilon},\: 0 < \epsilon \ll 1.$ Moreover, in \cite{V1} an approximation for $\nc(- \frac{1}{\th})$ has been obtained by a semiclassical pseudo-differential matrix valued operator.\\

We deal with the elliptic case, where $\theta  \geq 1 - h^2$. In this case according to the results in \cite{V1}, an approximation $T_N(h, z)(\nu \wedge f)$ of $\nc(- \frac{1}{\th})f$ can be constructed with a remainder having norm ${\mathcal O}(h^N)\| f\|^2$ choosing  $N \in \N$ very large. The principal symbol of $T_N(h, z)$ has matrix symbol $m = \frac{1}{z}\Bigl( \rho I + \frac{\bc}{\rho}\Bigr)$ (see Section 3), where $\bc$ is a symmetric matrix, $\rho = \sqrt{z^2 - r_0}$ and  $r_0(x', \xi')$ is the principal symbol of Laplace Beltrami operator $-h^2\Delta\vert_{\Gamma}$. To approximate $C(-\frac{1}{\th})$, we use the self-adjoint operator $\pc(h) = -T_N(h, -\ii) - \frac{1}{\gamma(x)}$. In  the case $\gamma_0(x) = \frac{1}{\gamma(x)} > 1, \: \forall x \in \Gamma,$ there exist values of $h$ for which $\pc(h)$ is not invertible. The semiclassical analysis of $\pc(h)$ is related to the eigenvalues of the principal symbol $-m- \gamma_0$ which has a double eigenvalue $\sqrt{1 + r_0} - \gamma_0$ and an eigenvalue $r(h) =(1 +  r_0)^{-1/2} - \gamma_0$. The symbol $r(h)$ is elliptic but $\lim_{|\xi'| \to \infty} r(h) = -\gamma_0$ and this leads to problems in the  semiclassical analysis of the spectrum of $\pc(h)$ (see Section 12 in \cite{DS} and hypothesis (H2)). To overcome this difficulty, we introduce a global diagonalisation of $m$ with a unitary matrix $U$ and write $(Op_h(U))^{-1} \pc(h) Op_h(U)$ in a block matrix form (see Section 4). We study the eigenvalues $\mu_k(h)$ of a self-adjoint operator $Q(h)$ and show that the invertibility of $Q(h)$ implies that of $\pc(h)$. This approach is more convenient that the investigation of ${\rm det}\: \pc(h)$. If $h_k, \: 0 < h_k \leq h_0,$ is such that  $\mu_k(h_k) = 0,$ then $Q(h_k)$ is not invertible and in this direction our analysis is very similar to that in \cite{SjV} and \cite{P2}. The next step is to express the trace formula involving $\pc(h)^{-1}$ with a trace one involving $Q(h)^{-1}$ (see Proposition 5.2). Finally, the problem is reduced to the  count of the negative eigenvalues of $Q(1/r)$ for $1/r < h_0.$ This strategy is not working if  $\gamma_0(x) = \gamma(x) > 1.$ To cover this case, we consider the problem (\ref{eq:3.10}) and the operator $\nc_1(- \frac{1}{\th})$ related to the solution of it.  Then we introduce the operators $C_1(h), \pc_1(h), Q_1(h)$. We study the eigenvalues of the self-adjoint operator $Q_1(h)$ and repeat the analysis in the case $0 < \gamma(x) <1.$ The eigenvalues of the semiclassical principal symbols of both operators  $Q(h)$ and $Q_1(h)$ are $\sqrt{1 + r_0} - \gamma_0$.\\
 
        The argument of our paper with technical complications can be applied to study the non homogenous Maxwell equations (see (\ref{eq:2.1}) for the notation) 
\begin{equation} \label{eq:1.4}
\begin{cases}\curl E =  - \lambda \mu(x) H,\: x \in \Omega,\\
 \curl H = \lambda \ep(x) E,\: \:x \in \Omega,  \\
\frac{1}{\gamma(x)} (\nu \wedge (\nu \wedge E)) + (\nu \wedge H)= 0\: {\rm for }\: x \in \Gamma,\\
(E, H): \ii\lambda-  {\rm outgoing}.
\end{cases}
\end{equation}        
Here $\ep(x) > 0,\: \mu(x) > 0$ are scalar valued functions in $C^{\infty}(\bar{\Omega})$ which are equal to constants $\ep_0,\: \mu_0$ for $|x| \geq c_0 > a.$ For this purpose it is necessary to generalise the results for eigenvalues free regions in \cite{CPR2} and to apply the construction in \cite{V1} concerning the non homogeneous case. \\        
        
   The paper is organised as follows. In Section 2 in  the case $0 < \gamma(x) < 1, \forall x \in \Gamma$ we introduce the operators $\nc(\lambda), C(\lambda),\: P(\lambda)$ and prove a trace formula (see Proposition 2.1). Similarly, in the case $\gamma(x) > 1,\forall x \in \Gamma,$ the operators $\nc_1(\lambda),\: C_1(\lambda), \pc_1(\lambda)$ are introduced. In Section 3 we collect some facts concerning the construction of a semiclassical parametrix for the problems (\ref{eq:2.8}), (\ref{eq:2.10}) build in \cite{V1}. Setting $\lambda = - \frac{1}{\th},\: \th = h ( 1 + \ii t), \: 0 < h \leq h_0$, we treat the case $z = \ii h \lambda =   -\ii(1 + \ii t)^{-1} $ with $|t| \leq h^2$ and this implies some simplifications. The self-adjoint operators $Q(h), Q_1(h)$ and their eigenvalues $\mu_k(h)$ are examined in Section 4. Finally, in Section 5 we compare the trace formulas  involving $\pc(h)$ and $\cc(h)$ and show that they differ by negligible terms. The proof of Theorem 1.1 is completed by the asymptotic of the negative eigenvalues of $Q(1/r), Q_1(1/r).$ 
   
 \section{Trace formula for Maxwell's equations}
\def\hn{h^{(1)}_n}

 An eigenfunction $D(G_b) \ni (E, H) \neq 0$ of $G_b$ with eigenvalue $\lambda \in \{z \in \C:\Re z < 0\}$
 satisfies
\begin{equation} \label{eq:2.1}
\begin{cases}\curl E =  - \lambda H,\: x \in \Omega,\\
 \curl H = \lambda E,\: \:x \in \Omega,  \\
\frac{1}{\gamma(x)} (\nu \wedge (\nu \wedge E)) + (\nu \wedge H)= 0\: {\rm for }\: x \in \Gamma,\\
(E, H): \ii\lambda-  {\rm outgoing}.
\end{cases}
\end{equation}
The $\ii \lambda$- outgoing condition means that every component of $E = (E_1, E_2, E_3)$ and $H = (H_1, H_2, H_3)$ satisfies the $\ii \lambda$-outgoing condition for the equation $(\Delta - \lambda^2)u = 0$, that is
$$\frac{d}{dr} \Bigl(E_j(r \omega)\Bigr)  -\lambda E_j(r \omega)  = {\mathcal O}\Bigl(\frac{1}{r^2}\Bigr) , \: j = 1, 2, 3,\:r \to \infty$$
uniformly with respect to $\omega \in S^2$ and the same condition holds for $ H_j, \: j = 1, 2, 3.$ This condition can be written in several equivalent forms and for Maxwell's equation it is known also as Silver-M\"uller radiation condition (see Remark 3.31 in \cite{KH}). Notice that we can present $E$ and $H$ by integrals involving the outgoing resolvent $(\Delta_0 - \lambda^2)^{-1} $ of the free Laplacian in $\R^3$ with kernel $R_0(x, y; \lambda) = \frac{e^{\lambda |x - y|}}{4 \pi |x - y|},\: x \neq y,$  and if $(E, H)$ satisfy the $\ii \lambda$- outgoing condition, we can apply the Green formula
$$\int_{\Omega} \dive (A \wedge B) dx = \int_{\Omega} (\la B, \curl A \ra - \la A, \curl B \ra) dx = \int_{\Gamma} \la( \nu \wedge B), A \ra dS_x,$$
where $\la . , . \ra$ is the scalar product in $\C^3.$\\

First we treat the case $0 < \gamma(x) < 1, \forall x \in \Gamma.$ The case $\gamma(x) > 1, \forall x \in \Gamma$ will be discussed at the end of this section. Introduce the spaces
$$\hc_s^t(\Gamma) = \{f \in H^s(\Gamma; \C^3): \: \langle \nu(x), f(x) \rangle = 0\}$$ and
 consider the boundary problem
 \begin{equation} \label{eq:2.2}
\begin{cases}\curl E =  -\lambda H + U_1,\: x \in \Omega,\\ 
 \curl H = \lambda E + U_2, \:x \in \Omega,  \\
\nu \wedge E= f\: {\rm for }\: x \in \Gamma,\\
(E, H): \ii \lambda-  {\rm outgoing} 
\end{cases}
\end{equation}
with $U_1, U_2  \in L^2(\Omega; \: \C^3),\: \dive U_1, \: \dive U_2 \in L^2(\Omega),\:f \in \hc_1^t(\Gamma).$
Consider the operator
$$\nc(\lambda) : \hc_1^t(\Gamma) \ni f \longrightarrow  \nu \wedge H\vert_{\Gamma} \in \hc_0^t(\Gamma),$$
$(E, H)$ being the solution of (\ref{eq:2.2}) with $U_1 = U_2 = 0.$  According to Theorem 3.1 in \cite{V1}, this operator is well defined and it plays the role of the Dirichlet-to-Neumann operator for the Helmoltz equation $(\Delta - \lambda^2)u = 0.$ By $\nc(\lambda)$, we write the boundary condition in (\ref{eq:2.1}) as follows
\begin{equation} \label{eq:2.3} 
C(\lambda)f: = \nc(\lambda)f + \frac{1}{\gamma(x)} (\nu \wedge f) = 0,\: \nu \wedge E\vert_{\Gamma} = f \in \hc_1^t(\Gamma).
\end{equation} 

Introduce the operator $P(\lambda) (f) := \nc(\lambda)(\nu \wedge f)$, that is $\nc(\lambda) f = - P(\lambda)(\nu \wedge f).$ Therefore,  since $\gamma_0(x) = \frac{1}{\gamma(x)}$ the condition (\ref{eq:2.3}) becomes
\begin{equation} \label{eq:2.4} 
\tilde{C}(\lambda)g : = P(\lambda) g- \gamma_0(x)  g   = 0,\: g = \nu \wedge f = - E_{tan} \vert_{\Gamma}
\end{equation} 
and $P(\lambda):\: \hc_1^t (\Gamma)\longrightarrow \hc_0^t(\Gamma).$ 
 For $\lambda \in \R^{-}$, it is easy to see that the operator $P(\lambda)$ is self-adjoint in $\hc_0^t.$ To do this, we must prove that for $u, v \in \hc_1^t (\Gamma)$ we have
\begin{eqnarray} \label{eq:2.5} 
-(P(\lambda) (\nu \wedge u), \nu \wedge v) = (\nc(\lambda) u, \nu \wedge v) \nonumber \\= (\nu \wedge u, \nc(\lambda) v) = -( \nu \wedge u, P(\lambda)( \nu \wedge v)),
\end{eqnarray} 
where $(. , .)$ is the scalar product in $\hc_0^t(\Gamma).$  
Let $(E, H)$ (resp. $(X, Y)$) be the solution of the problem (\ref{eq:2.2}) with $U_1 = U_2 = 0$ and $f$ replaced by $u$ (resp. $v$). By applying the Green formula, we get 
$$\lambda \int_{\Omega} (\la E, X \ra + \la Y, H \ra )dx = \int_{\Omega} \la E, \curl Y \ra dx - \int_{\Omega}\la Y, \curl E \ra dx $$
$$= -\int_{\Gamma} \la \nu \wedge Y, E \ra =  \int_{\Gamma} \overline{\la Y, u \ra}=  \int_{\Gamma} \overline{\la \nu \wedge Y, \nu \wedge u\ra} = (\nu \wedge u, \nc(\lambda) v).$$
Similarly,
$$\lambda \int_{\Omega} ( \la X, E \ra + \la H, Y \ra)  dx = \int_{\Omega} \la X, \curl H \ra dx - \int_{\Omega}\la H, \curl X \ra dx $$
$$= -\int_{\Gamma} \la \nu \wedge H, X \ra =  \int_{\Gamma} \overline{\la H, v \ra} = \int_{\Gamma} \overline{\la \nu \wedge H,  \nu \wedge v \ra}  =  \overline{( \nc(\lambda) u, \nu \wedge v) }$$
and for real $\lambda$  we obtain (\ref{eq:2.5}).\\

Let $(E, H) = R_C(\lambda) (U_1, U_2)$ be the solution of (\ref{eq:2.2}) with $f = 0.$ Then $R_C(\lambda) = (G_C - \lambda)^{-1}$, where $G_C$ is the operator $G$ with boundary condition $\nu \wedge E\vert_{\Gamma} = 0$ and domain $D(G_C) \subset \hc.$ The operator $\ii G_C$ is self-adjoint and $(G_C - \lambda)^{-1}$ is analytic operator valued function for $\Re \lambda < 0.$ On the other hand, it is easy to express $\nc(\lambda)$ by $R_C(\lambda).$ Given $f \in \hc_1^t(\Gamma)$, let $e_0(f) \in H^{3/2} (\Omega; \C^3)$ be an extension of $-(\nu \wedge f)$ with compact support. Consider
$$\begin{pmatrix} u \\v \end{pmatrix} = - R_C(\lambda)\Bigl( ( G - \lambda) \begin{pmatrix} e_0(f) \\ 0 \end{pmatrix} \Bigr)+ \begin{pmatrix} e_0(f) \\ 0 \end{pmatrix}. $$
Then $(u, v)$ satisfies (\ref{eq:2.2}) with $U_1 = U_2 = 0$ and $\nc(\lambda)f = \nu \wedge v\vert_{\Gamma}$ implies that $\nc(\lambda)$ is analytic for $\Re \lambda < 0.$ Consequently, $C(\lambda) : H_1^t(\Gamma)  \rightarrow H_0^t(\Gamma)$ is also analytic for $\Re \lambda< 0.$ On the other hand, for $\Re \lambda < 0$ the operator $\nc(\lambda)$ is invertible. Indeed, if $\nc(\lambda) f = 0,$ let $(E, H)$ be a solution of the problem
\begin{equation} \label{eq:2.6}
\begin{cases}\curl E =  -\lambda H,\: x \in \Omega,\\ 
 \curl H = \lambda E, \:x \in \Omega,  \\
\nu \wedge H= 0\: {\rm for }\: x \in \Gamma,\\
(E, H): \ii \lambda-  {\rm outgoing}.
\end{cases}
\end{equation}
 By Green formula one gets
$$\bar{\lambda} \int_{\Omega}( |E|^2 + |H|^2) dx =  \int_{\Omega} \la E, \curl H \ra dx- \int_{\Omega} \la H, \curl E \ra dx = -\int_{\Gamma}\la  \nu \wedge H, E \ra = 0.$$
This implies $E = H = 0$, hence $ f = 0.$ Thus we conclude that  for $\Re\lambda < 0$ the operator $\nc(\lambda)^{-1}$ is analytic and 
$$C(\lambda)f  = \nc(\lambda)\Bigl ( I + \nc(\lambda)^{-1} \gamma_0(x)  i_{\nu} \Bigr) f. $$
Here $i_{\nu}(x)$ is a $(3 \times 3)$ matrix such that  $i_{\nu} (x)f = \nu(x) \wedge f.$ 
The operator $\nc(\lambda)^{-1}: H_0^t (\Gamma) \rightarrow H_1^t (\Gamma)$ is compact and by the analytic Fredholm theorem one deduces that
$$ C (\lambda)^{-1}   = \Bigl ( I + \nc(\lambda)^{-1} \gamma_0(x) i_{\nu} \Bigr)^{-1} \nc (\lambda)^{-1}  $$
is a meromorphic operator valued function.

    To establish a trace formula involving $( G_b -\lambda)^{-1} $, consider
    $(G_b -\lambda) (u, v) = (F_1, F_2) = X,\: (u, v) \in D(G_b).$ Then
    $$ \begin{cases} \curl u = -\lambda v + F_2, \: x \in \Omega,\\
    \curl v = \lambda u + F_1,\: x \in \Omega \end{cases}$$
and $(u, v) =  R_C(\lambda) X  +  K(\lambda) f = (G_b - \lambda)^{-1} X,$
where $K(\lambda)f$ is solution of  (\ref{eq:2.2}) with $U_1 = U_2 = 0$. Let $R_C(\lambda) X = ((R_C(\lambda) X)_1, (R_C(\lambda) X)_2).$ Notice that for $\Re\lambda < 0$, $(R_C(\lambda)X)_j,\: j = 1, 2 ,$ are analytic vector valued functions. To satisfy the boundary condition, we must have
$$\gamma_0(x)(\nu \wedge f )+ \Bigl(\nc(\lambda) (f )+  \nu \wedge (R_C(\lambda) X )_2\vert_{\Gamma} \Bigr) = 0,$$
hence 
$$f =   -C(\lambda)^{-1} \Bigl(\nu \wedge (R_C(\lambda)X)_2\vert_{\Gamma}\Bigr),$$
provided that $C(\lambda)^{-1}$ exists.

Assuming that $C(\lambda)^{-1}$ has no poles on a closed positively oriented curve $\delta \subset \{z \in \C: \: \Re z < 0\}$, we apply Lemma 2.2  in \cite{SjV} and exploit the cyclicity of  the trace to conclude that the operators
$$ \mathcal H \ni X \longrightarrow  -\int_{\delta}  K(\mu)\Bigl(  C(\lambda)^{-1} \Bigl(\nu \wedge (R_C(\lambda)X)_2\vert_{\Gamma}\Bigr)\Bigr) d\lambda \in \mathcal H $$
and
$$H_1^1(\Gamma) \ni w  \longrightarrow  -\int_{\delta} C(\lambda)^{-1} \Bigl(\nu \wedge  \Bigl(R_C(\lambda)(K(\lambda) (w))\Bigr)_2\big\vert_{\Gamma}\Bigr)d\lambda \in H_1^t(\Gamma)$$
have the same traces. 
On the other hand, 
$$\begin{cases} (G - \lambda) \frac{\partial K(\lambda) (w) }{\partial \lambda} =  K(\lambda)(w),\\
\nu \wedge \Bigl(\frac{\partial K(\lambda) (w)}{\partial \lambda}\Bigr)_1\big\vert_{\Gamma} = 0.\end{cases}.$$
This implies $\Bigl(R_C(\lambda) (K(\lambda) (w) )\Bigr)_2\big\vert_{\Gamma} = \Bigl(\frac{\partial K(\lambda) (w)}{\partial \lambda}\Bigr)_ 2\big\vert_{\Gamma} $
and
$$\nu \wedge \Bigl(\frac{\partial K(\lambda) (w)}{\partial \lambda}\Bigr)_ 2\big\vert_{\Gamma} =  \frac{\partial \nc(\lambda) w}{\partial \lambda} =  \frac{\partial C(\lambda)(w)}{\partial \lambda}. $$
The integrals involving the analytic terms $(R_C(\lambda) X)_j, \: j = 1, 2,$ vanish and we obtain the following
\begin{prop}
Let $0 < \gamma(x) < 1,\:\forall x \in \Gamma$ and let $\delta \subset \{z \in \C: \: \re z < 0\}$ be a closed positively oriented curve without self intersections such that $ C(\lambda)^{-1}$ has no poles on $\delta.$ Then
\begin{equation} \label{eq:2.7}
\tr_{\mathcal H} \frac{1}{2 \pi \ii} \int_{\delta} (\lambda - G_b)^{-1} d\lambda = \tr_{\mathcal H_1^t(\Gamma) } \frac{1}{2 \pi \ii} \int_{\delta}  C(\lambda)^{-1} \frac{d  C(\lambda)}{d \lambda} d\lambda. 
\end{equation}
\end{prop}
The left hand side of (\ref{eq:2.7}) is equal to the number of the eigenvalues of $G_b$ in the domain bounded by $\delta$ counted with their multiplicities.
Set $\lambda = - \frac{1}{\tilde{h}}$ with  $0 < \re \tilde{h} \ll 1$. For $\lambda \in \Lambda$ we have $\tilde{h} \in L,$ where
$$L: = \{\tilde{h} \in \C: \: |\im \tilde{h} | \leq C_1 |\tilde{h}|^4,\: |\tilde{h}| \leq C_0^{-1}, \: \re \tilde{h} > 0\}.$$
Write $\tilde{h} = h( 1 + \ii t),\: 0 < h \leq h_0 \leq C_0^{-1},\: t \in \R.$ Then for $\th \in L$, it is easy to see that $|t| \leq h^2$ for $\tilde{h} \in L$
and the problem (\ref{eq:2.2}) with $U_1 = U_2 = 0$  becomes
\begin{equation} \label{eq:2.8}
\begin{cases}-\ii  h\: \curl E =  z H,\: x \in \Omega,\\
- \ii h\: \curl H  = - z E, \: x \in \Omega,\\
 \nu \wedge E = f,\: x \in \Gamma,\\
(E, H) - {\rm outgoing} \end{cases}
\end{equation} 
with $-\ii z = h \lambda, \:z = -\ii ( 1+ \ii t)^{-1}.$ We introduce the operator $ \cc(\th) = C(- \th^{-1})$ and the trace formula is transformed in
\begin{equation} \label{eq:2.9} 
\tr_{\mathcal H} \frac{1}{2 \pi \ii} \int_{\delta} (\lambda - G_b)^{-1} d\lambda = \tr_{\mathcal H_1^t(\gamma) } \frac{1}{2 \pi \ii} \int_{\tilde{\delta}}  \cc(\th)^{-1} \frac{d \cc(\th)}{d \th} d\th, 
\end{equation}
where $\tilde{\delta} = \{ z \in \C: \: z = - \frac{1}{w},\: w \in \delta\}.$\\

    To deal with the case $\gamma(x) > 1, \: \forall x \in \Gamma$, we write the boundary condition in (\ref{eq:1.1}) in the form
    $$- (\nu \wedge E) + \gamma_0(x) (\nu \wedge (\nu \wedge H)) = 0,\: x \in \Gamma.$$
Consider the boundary problem
 \begin{equation} \label{eq:2.10}
\begin{cases}\curl E =  -\lambda H,\: x \in \Omega,\\ 
 \curl H = \lambda E, \:x \in \Omega,  \\
\nu \wedge H= f\: {\rm for }\: x \in \Gamma,\\
(E, H): \ii \lambda-  {\rm outgoing} 
\end{cases}
\end{equation}
and introduce the operator
$$\nc_1(\lambda) : \hc_1^t(\Gamma)\ni f \longrightarrow  \nu \wedge E\vert_{\Gamma} \in \hc_0^t(\Gamma),$$
where $(E, H)$ is the solution of (\ref{eq:2.10}). The above boundary condition becomes
$$C_1(\lambda) : = \nc_1(\lambda) f - \gamma_0(x) (\nu \wedge f) = 0,\: f = \nu \wedge H\vert_{\Gamma}.$$
Now we introduce the operator 
$P_1(\lambda) f= - \nc_1(\lambda)(\nu \wedge f)$  and write the boundary condition as
\begin{equation} \label{eq:2.11}
\tilde{C}_1(\lambda)f: = P_1(\lambda) g - \gamma_0(x) g = 0, \: x \in \Gamma ,\: g =  -H_{tan} \vert_{\Gamma}.
\end{equation} 
Comparing (\ref{eq:2.11}) with (\ref{eq:2.4}), we see that both boundary conditions are writen  by $\gamma_0(x).$  Clearly, we may repeat the above argument and obtain a trace formula involving $C_1(\lambda)^{-1} $ and $\frac{d}{d\lambda} C_1(\lambda).$

 \section{Semiclassical parametrix in the elliptic region}
 In this section we will collect some results in \cite{V1} concerning the construction of a semi-classical parametrix of the problem (\ref{eq:2.8}) and we refer to this work for more details.
 Let $\theta = |\im z| = \frac{1}{1 + t^2} \leq 1.$ Then the condition $\theta > h^{2/5 - \ep},\: 0 < \ep \ll 1$ in \cite{V1} is trivially satisfied for small $h_0.$ Moreover, $\theta \geq 1 - t^2 \geq 1 - h_0^2$ so $\theta$ has lower bound independent of $h$. This simplifies the construction in \cite{V1}. In the exposition we will use $h$-pseudo-differential operators
 and we refer to \cite{DS} for more details.
  Let $(x_1, x')$ be local geodesic coordinates in a small neighbourhood $\mathcal U \subset \R^3$ of $y_0 \in \Gamma$. We set $x_1 = \dist(y, \Gamma),\: x'= s^{-1}(y)$, where $x'= (x_2, x_3)$ are local coordinates in a neighborhood $\mathcal U_0 \subset \R^2$ of (0, 0)  and $s: {\mathcal U}_0 \rightarrow \mathcal U \cap \Gamma$  is a diffeomorphism. Set $\nu(x') = \nu(s(x')) =(\nu_1(x'), \nu_2(x'), \nu_3(x')).$   Then $y = s(x') + x_1 \nu(x')$
  and (see Section 2 in \cite{CPR2} and  Section 2 in \cite{V1})
    $$\frac{\partial}{\partial y_j} = \nu_j(x') \frac{\partial}{\partial x_1} + \sum_{k = 2}^3 \alpha_{j, k} (x) \frac{\partial}{\partial x_k},\: j = 1, 2, 3.$$
  The functions $\alpha_{j, k}(x)$ are determined as follows.
  Let $\zeta_1 = (1, 0, 0),\:\zeta_2 = (0, 1, 0), \: \zeta_3 = (0, 0, 1)$ be the standard orthonormal basis in $\R^3$ and let $d(x)$ be a smooth matrix valued function such that 
  $$d(x) \zeta_1 = \nu(x'), d(x) \zeta_k = (\alpha_{1, k}(x), \alpha_{2, k}(x), \alpha_{3, k}(x)), \: k = 2, 3$$
  and $\nabla_y = d(x) \nabla_x.$
   Denote by $\xi= (\xi_1, \xi')$ the dual variables of $(x_1,x')$. Then the symbol of the operator $- \ii\: \nabla\vert_{x_1 = 0}$ in the coordinates $(x, \xi)$ has the form $ \xi_1 \nu(x') + \beta(x', \xi'),$  
  where $\beta(x', \xi')$ is vector valued symbol given by
  $$\beta(x', \xi') = \sum_{k = 2}^3 \xi_k d(0, x') \zeta_k= \Bigl(\sum_{k= 2}^3 \xi_k \alpha_{j, k}(x) \Bigl)_{j = 1, 2, 3}$$
  and $\la \nu(x'), \beta(x', \xi')\ra = 0.$ The principal symbol of the operator $-\Delta\vert_{x_1 = 0}$ becomes
  $$\xi_1^2 + \la \beta(x', \xi') , \beta(x', \xi') \ra,$$
  while the principal symbol of the Laplace-Beltrami operator $- \Delta_{\Gamma}$ has the form 
  $$r_0(x', \xi') = \la \beta(x', \xi'), \beta(x', \xi')\ra.$$
  It is important to note that $\beta(x', \xi')$ is defined globally and it is invariant when we change the coordinates $x'$. In fact if $\tilde{x}'$ are new coordinates, $ x_1 = \tilde{x}_1$  and
  $$y = \tilde{s}(\tilde{x}') + x_1 \nu(\tilde{x})= s(x') + x_1 \nu(x'),$$
  in the intersection of the domains $s(\mathcal U_0) \cap \tilde{s}(\tilde{\mathcal U}_0),$ where the coordinates $x'$ and $\tilde{x}'$ are defined, then $\nu(x') = \nu(\tilde{x}').$ From the equality $\nabla\vert_{x_1 = 0} = \nabla\vert_{\tilde{x}_1 = 0},$ we deduce
  $\beta(x', \xi') = \beta (\tilde{x}', \tilde{\xi}').$
  
    Let $a(x, \xi'; h) \in C^{\infty}(T^*(\Gamma) \times (0, h_0]).$ 
 Given $k \in \R, 0 < \delta < 1/2$, denote by ${\mathcal S}^{k}_{\delta}$ the set of symbols so that 
$$
\big |\pa_{x'}^{\alpha} \pa_{\xi'}^{\beta} a(x, \xi'; h)\big| \leq C_{\alpha, \beta} h^{-\delta(|\alpha| + |\beta|}  \la \xi' \ra^{k - |\beta|},\: \forall \alpha, \forall\beta,\quad
  (x', \xi) \in T^*(\Gamma)
 $$
 with $\la \xi ' \ra = ( 1 + |\xi'|^2)^{1/2}$ and constants $C_{\alpha, \beta}$ independent of $h$. 
A matrix symbol $m$ belongs to ${\mathcal S}^k_{\delta}$ if all entries of $m$ are in the class ${\mathcal S}^k_{\delta} .$ The $h-$pseudo-differential operator with symbol $a(x, \xi; h)$ acts by 
$$
(Op_h(a) f)(x) :\:=\
 (2 \pi h)^{-2}\iint_{T^*(\Gamma)} e^{\ii \langle y' - x ', \xi\rangle /h} a(x, \xi'; h) f(y) d \xi' dy'.
$$
By using the change $\xi'= h \eta'$, the operator can be written also as a classical pseudo-differential operator 
$$
(Op_h(a) f)(x) :\:=\
 (2 \pi)^{-2}\iint_{T^*(\Gamma)} e^{\ii \langle y' - x ', \xi\rangle} a(x, h\xi'; h) f(y) d \xi' dy'.
$$
Next for a positive function $\omega( x', \xi')  > 0$ we define the space of symbols $a(x', \xi';  h) \in S^k_{\delta_1, \delta_2}(\omega)$ for which 
$$
\big |\pa_{x'}^{\alpha} \pa_{\xi'}^{\beta} a(x, \xi'; h)\big | \leq C_{\alpha, \beta} \omega^{k -\delta_1|\alpha| - \delta_2 |\beta|} ,\: \forall \alpha, \forall\beta,\quad
  (x', \xi') \in T^*(\Gamma).
 $$
We denote $S^k_{\delta_1, \delta_2} = S^k_{\delta_1, \delta_2} (\la \xi'\ra)$ and introduce the norm 
$$\|u\|_{H^k_h(\Gamma)} : = \|Op_h( \la \xi'\ra^k) u \|_{L^2(\Gamma)}.$$
 Let $\rho(x', \xi', z) = \sqrt{z^2 - r_0(x', \xi')},\: \im \rho > 0 ,$ be the root of the equation $ \xi_1^2 + r_0(x', \xi') - z^2 = 0$ with respect to $\xi_1.$ Set $z = -\ii + 	t(h),\: t(h) = {\mathcal O}(h^2).$ We have $\rho \in \mathcal S^1_{0}$,
 $$\sqrt{z^2 - r_0} = \ii \sqrt{1 + r_0}- \frac{{\mathcal O}(h^2)}{\ii \sqrt{1 + 2 \ii t(h) - t^2(h) + r_0} + \ii \sqrt{ 1+ r_0}}$$
 and $\rho - \ii \sqrt{1 + r_0} \in \mathcal S^{-1}_0.$\\
 
      The local parametrix of (\ref{eq:2.8}) constructed in \cite{V1} in local coordinates $(x_1, x')$ has the form
 $$\tilde{E} = (2\pi h)^{-2} \iint e^{\frac{\ii}{h}(\la y', \xi'\ra  + \varphi(x, \xi', z))} \phi_0^2 (x_1/ \delta) a(x, y', \xi', z, h) d\xi'dy',$$  
  $$\tilde{H} = (2\pi h)^{-2} \iint e^{\frac{\ii}{h}(\la y', \xi'\ra  + \varphi(x, \xi', z))} \phi_0^2 (x_1/ \delta) b(x, y', \xi', z, h) d\xi'dy',$$ 
 where $\phi_0(s) \in C_0^{\infty}(\R)$ is equal to 1 for $|s| \leq 1$ and to 0 for $|s| \geq 2$  and $0 < \delta \ll 1.$ Set $\chi(x_1) = \phi_0^2(x_1/\delta).$ The phase function $\varphi$ 
 satisfies for $N$ large the equation 
 $$\la d\nabla_x \varphi, d\nabla_x \varphi\ra - z^2 \varphi = x_1^N \Phi$$
 and has the form 
 $$\varphi = \sum_{k= 0}^{N-1} x_1^k \varphi_k( x', \xi', z), \: \varphi_0 = - \la x', \xi'\ra,\: \varphi_1 = \rho.$$
 Moreover, for $\delta$ small enough we have $\im \varphi \geq x_1 \im \rho/2$ for $0 \leq x_1 \leq 2\delta.$
  The construction of $\varphi$ is given in \cite{V}, \cite{V1}. For $ z = -\ii$ we have $\varphi = - \la x'.\xi'\ra + \ii \tilde{\varphi}$ with real valued phase $\tilde{\varphi}$ (see \cite{V} and Section 3 in \cite{P2}). 
  Introduce  a function $\eta \in C^{\infty}(T^*(\Gamma))$ such that $\eta = 1$ for $ r_0 \leq C_0,\: \eta = 0$ for $r_0 \geq 2C_0$, where $C_0 > 0$ is independent on $h$. Choosing $C_0$ big enough, one arranges the estimates
  \begin{equation} \label{eq:3.1} 
  \begin{aligned} 
  C_1 \leq |\rho| \leq C_2,\: \im \rho \geq C_3, \: (x', \xi') \in {\rm supp}\: \eta,\\
  |\rho| \geq \im \rho \geq C_4|\xi'|,\: (x', \xi') \in {\rm supp}\: (1 - \eta)
  \end{aligned} 
  \end{equation} 
  with positive constants $C_j > 0.$  Following \cite{V1}, we say that a symbol $\omega \in C^{\infty}(T^*(\Gamma))$ is in the class $S^{k_1}_{\delta_1, \delta_2} (\omega_1) + S^{k_2}_{\delta_3, \delta_4} (\omega_2)$  if $\eta \omega \in S^{k_1}_{\delta_1, \delta_2} (\omega_1)$ and $(1 - \eta) \omega \in S^{k_2}_{\delta_3, \delta_4} (\omega_2).$ The amplitudes $a$ and $b$ have the form $a = \sum_{j =0}^{N- 1} h^j a_j,\: b = \sum_{j= 0}^{N- 1} h^j b_j$ and $a_j, b_j$ for $0 \leq j \leq N- 1$ satisfy the system
  \begin{equation} \label{eq:3.2} 
 \begin{cases} (d\nabla_x \varphi) \wedge a_j - z b_j = \ii(d \nabla_x) \wedge a_{j-1} + x_1^N \Psi_j,\\
 (d \nabla_x \varphi) \wedge b_j + z a_j = \ii (d \nabla_x) \wedge b_{j-1} + x_1^N \tilde{\Psi}_j,\\
 \nu \wedge a_j = \begin{cases} g, \: j = 0,\\ 0, \: j \geq 1 \end{cases}  \: {\rm on}\: x_1 = 0,
 \end{cases}
 \end{equation}  
 where $a_{-1} = b_{-1} = 0$ and 
 $$g=  - \nu(x') \wedge ( \nu(y') \wedge f(y')) = f(y') - (v(x') - v(y')) \wedge (\nu(y') \wedge f(y')).$$ 
  On the other hand,  the function $a_j, b_j$ have the presentation
  $$a_j = \sum_{k= 0}^{N- 1} x_1^k a_{j, k},\: b_j = \sum_{k = 0}^{N-1} x_1^k b_{j, k}.$$
  The symbols $a_{j, k}, b_{j, k}$ are expressed by terms involving $g$. Moreover,
  $$a_{j, k} = A_{j, k}(x', \xi') \tilde{f}(y'),\: b_{j, k} = B_{j, k}(x', \xi') \tilde{f}(y'),$$
  where $\tilde{f}(y') = \nu(y') \wedge f(y') = i_{\nu}(y') f(y')$ with a $(3 \times 3)$ matrix $i_{\nu} = \sum_{j = 1}^3 \nu_j I_j$, $I_j$ being  $(3 \times 3)$ constant matrices. Here $A_{j, k}, B_{j, k}$ are smooth matrix valued functions.
  The important point proved in Lemma 4.3 in \cite{V1} is that we have the properties
  \begin{equation} \label{eq:3.3} 
  \begin{aligned}
  A_{j, k} \in S^{-1 -3k-  5j}_{2,2}(|\rho|) + S^{-j}_{0, 1}(|\rho|),\: j \geq 0, \: k \geq 0,\\
  B_{j, k} \in S^{-1 -3k- 5j}_{2,2}(|\rho|) + S^{1-j}_{0, 1}(|\rho|),\: j \geq 0, \: k \geq 0.  
  \end{aligned}
\end{equation} 
  Since by (\ref{eq:3.1}), the function $|\rho|$ is bounded from below for $(x', \xi') \in {\rm supp} \: \eta$, in the above properties we may replace absorb $S^{-1-3k-5j}_{2,2}(|\rho|) $ and obtain the class  $S^{-j}_{0, 1}(|\rho|)$ (resp. $S^{1- j}_{0, 1}(|\rho|) $) for all $(x', \xi')$. For the principal symbols $a_{0, 0},\: b_{0,0}$ we have form (\ref{eq:3.2}) the system
  \begin{equation} \label{eq:3.4}
  \begin{cases} \psi_0 \wedge a_{0, 0} - z b_{0, 0} = 0,\\
  \psi_0 \wedge b_{0, 0} + z a_{0,0} = 0,\\
  \nu \wedge a_{0, 0} = g,
  \end{cases}
  \end{equation} 
  with $\psi_0 = d(0, x') \nabla_x \varphi\vert_{x_1 = 0} = \rho \nu - \beta.$ The solution of (\ref{eq:3.4}) is given by (4.4) in \cite{V1} and one has
  \begin{equation}\label{eq:3.5} 
  \begin{aligned}
  a_{0, 0} = - \nu \wedge g + \rho^{-1} \la \nu, \beta \wedge g\ra \nu,\\
  \nu \wedge b_{0, 0} = \frac{1}{z} \Bigl( \rho (\nu \wedge g )+ \rho^{-1} \la\beta , \nu \wedge g\ra \beta\Bigr). 
  \end{aligned} 
  \end{equation} 
  Thus we obtain $\nu \wedge \tilde{E}\vert_{x_1 = 0} = f$ and
  $$\nu \wedge \tilde{H}\vert_{x_1 = 0}  = i_{\nu}(x') \tilde{H}\vert_{x_1 = 0} = \sum_{j = 0}^{N -1} h^j Op_h(i_{\nu} B_{j, 0} )\tilde{f}.$$
 Following \cite{V1} and using (\ref{eq:3.5}), for the principal symbol of $\nu \wedge \tilde{H}\vert_{x_1 = 0}$ one deduces
  $$i_{\nu} B_{0, 0} \tilde{f} = \nu \wedge b_{0,0} = m (\nu \wedge g) = m \tilde{f} + m i_{\nu} \sum_{j= 1}^3 (\nu_j(y') - \nu_j (x')) I_j \tilde{f}$$
  with a matrix symbol $m : = \frac{1}{z} (\rho I + \rho^{-1} \bc)$ and matrix valued symbol $\bc$ defined by $\bc v = \la \beta, v \ra \beta, \: v \in \R^3.$
  Then we obtain
  $$Op_h(i_{\nu} B_{0,0}) \tilde{f} = Op_h(m) \tilde{f} + h Op_h(\tilde{m} )\tilde{f}$$
  with $\tilde{m} \in \mathcal S^0_{0}.$  Choosing $\tilde{f} = \psi(x') (\nu \wedge f)$, we obtain a local parametrix $T_{N, \psi}(h, z)$ and in Theorem 1.1 in \cite{V1} the  estimate 
   \begin{equation} \label{eq:3.6} 
    \|\nc(\lambda)(\psi f)-  Op_h(m  + h \tilde{m})(\nu \wedge \psi f) \|_{\hc_0^t} \leq C h\theta^{-5/2} \|f\|_{\hc_{-1}^t}
    \end{equation} 
 has been  established in a more general setting assuming a lower bound $\theta >h^{2/5}.$ With the last condition one can study the case $ z = 1 + \ii \theta = h \lambda,\: h = |\Re \lambda|^{-1}$, provided  $|\Re \lambda| \geq |\im \lambda|$.   
  
 In this paper we need a parametrix  in the elliptic case $ z = - \ii + t(h)$ and in (\ref{eq:3.6}) we can obtain an approximation modulo ${\mathcal O}(h^{- \ell_2+ N})$  adding lower order terms of $T_{ N, \psi}(h, z)$ and exploiting the bound $\theta \geq 1 - h^2$ as well as the estimates (\ref{eq:3.1}), (\ref{eq:3.3}). According to Lemma 4.2 in \cite{V1}, one has the estimates
 $$\big | \partial_{x'}^{\alpha} \partial_{\xi'}^{\beta} \Bigl(e^{\frac{\ii \varphi}{h}}\Bigr)\big | \leq  C_{\alpha, \beta} |\xi'|^{-|\beta|} e^{- C |\xi'| x_1/h}$$
 for $0 \leq x_1 \leq 2 \delta$ with constants $C > 0, C_{\alpha, \beta} > 0$ independent of $x_1, z$ and $h.$ In fact, the above estimates are proved for $(x',\xi') \in {\rm supp}\: (1 - \eta)$, while for $(x', \xi') \in {\rm supp}\: \eta$ the factor $|\xi'|$ is bounded.  Then (see (4.31) in \cite{V1}) 
 $$h^{-N} x_1^N e^{\ii \varphi/h} \in S^{-N}_{0, 1} $$
 uniformly in $x_1$ and $h$. Now let
 $$-\ii h \nabla \wedge \tilde{E}- z \phi \tilde{H}  = (2 \pi h)^{-2} \iint e^{\frac{\ii}{h} (\la y', \xi'\ra + \varphi)} V_1(x, y', \xi', h, z) d\xi'dy' = U_1,$$
 $$-\ii h \nabla \wedge \tilde{H} + z \phi \tilde{E} = (2 \pi h)^{-2} \iint e^{\frac{\ii}{h} (\la y', \xi'\ra + \varphi)} V_2(x, y', \xi', h, z) d\xi'dy' = U_2$$
with
$$V_1 = h \tilde{\chi} a + h^N \chi(d \nabla_x) \wedge a_{N-1} + x_1^N \sum_{j=0}^{N-1} h^j \chi \Psi_j,$$
$$V_2 =  h \tilde{\chi} b + h^N \chi(d \nabla_x) \wedge b_{N-1} + x_1^N \sum_{j=0}^{N-1} h^j \chi \tilde{\Psi}_j,$$
where $\tilde{\chi}$ has support in $\delta \leq x_1 \leq 2 \delta.$ Clearly,
$$(h\partial_x)^{\alpha} U_j ( x_1, .) = Op_h\Bigl(h^{|\alpha|} \partial_x^{\alpha} (e^{\ii \varphi/h} V_j)\Bigr) \tilde{f}, \: j = 1, 2.$$
 Combing this with the properties (\ref{eq:3.3}) and the proof of Lemma 4.3 in \cite{V1}, we obtain the estimate
 \begin{equation}
 \|(h \partial_x )^{\alpha} U_j\|_{L^2(\Gamma)} \leq C_{\alpha, N} h^{- \ell_{\alpha} + N} \|f\|_{H_h^{-1}(\Gamma)}
 \end{equation}
 with $\ell_{\alpha}$ independent of $h, N$ and $f$.
Thus by the argument in Section 4 in \cite{V1} we construct a local parametrix in the elliptic region  and
 \begin{equation} \label{eq:3.7}
  \|\nc\Bigl(-\frac{1}{\th}\Bigr) (\psi f) - T_{N, \psi}(h, z) (\nu \wedge \psi f)\|_{H^{s}_h(\Gamma) } \leq C_{N} h^{- \ell_s + N} \|f\|_{L^2(\Gamma)},\: s \geq 0,\:  N \geq N_s.
  \end{equation}
 
   Choosing a partition of unity $\sum_{j = 1}^M \psi_j(x') \equiv 1$ on $\Gamma$, we construct a parametrix $T_N(h, z) = \sum_{j=1}^M T_{N, \psi_j} (h, z)$  and obtain
     \begin{equation} \label{eq:3.8}
  \|\nc\Bigl(-\frac{1}{\th}\Bigr) f- T_N(h, z) (\nu \wedge  f)\|_{H^{s}_h(\Gamma)} \leq C_{N} h^{-\ell_s+ N} \|f\|_{L^2(\Gamma) },\:s\geq 0,\:  N \geq N_s. 
  \end{equation} 
 
 For the operator $P(-\frac{1}{\th}) f = \nc(-\frac{1}{\th}) (\nu \wedge f)= \nc(\lambda)(\nu \wedge f)$
 one has an approximation by $- T_N(h, z)f$. Moreover,   
     for $ z = -\ii$ the principal symbol of $-T_N(h, -\ii)$ becomes
  $$- m = \frac{1}{\ii} \Bigl( \ii\sqrt{1 + r_0} I+ \frac{\bc}{\ii \sqrt{1 + r_0}}\Bigr) = \sqrt{1 + r_0} I - \frac{1}{\sqrt{1 + r_0}} \bc.$$
  
  Now we discuss briefly the existence of the parametrix for the problem
  \begin{equation} \label{eq:3.10}
\begin{cases}-\ii  h\: \curl E =  z H,\: x \in \Omega,\\
- \ii h\: \curl H  = - z E, \: x \in \Omega,\\
 \nu \wedge H = f,\: x \in \Gamma,\\
(E, H) - {\rm outgoing} \end{cases}
\end{equation} 
with $-\ii z = h \lambda, \:z = -\ii ( 1+ \ii t)^{-1}.$ We follow the construction above with the same phase function. The transport equations for $a_j, b_j$ have the form
 \begin{equation} \label{eq:3.11} 
 \begin{cases} 
 (d \nabla_x \varphi) \wedge b_j + z a_j = \ii (d \nabla_x) \wedge b_{j-1} + x_1^N \Psi_j,\\
 (d\nabla_x \varphi) \wedge a_j - z b_j = \ii(d \nabla_x) \wedge a_{j-1} + x_1^N \tilde{\Psi}_j,\\
 \nu \wedge b_j = \begin{cases} g, \: j = 0,\\ 0, \: j \geq 1 \end{cases}  \: {\rm on}\: x_1 = 0,
 \end{cases}
 \end{equation}  
 where $a_{-1} = b_{-1} = 0.$ This system is the same as (\ref{eq:3.2}) if we replace $z$ by $- z$ and $a_j, b_j$ by $b_j, a_j$, respectively. Therefore, by using (\ref{eq:3.5}), we obtain
 \begin{equation}\label{eq:3.12} 
  \begin{aligned}
  b_{0, 0} = - \nu \wedge g + \rho^{-1} \la \nu, \beta \wedge g\ra \nu,\\
  \nu \wedge a_{0, 0} = -\frac{1}{z} \Bigl( \rho (\nu \wedge g )+ \rho^{-1} \la\beta , \nu \wedge g\ra \beta\Bigr). 
  \end{aligned} 
  \end{equation}  
We obtain an analog of (\ref{eq:3.8}) with $\nc(h), T_N(h, z)$ replaced by $\nc_1(h), T_{1, N}(h, z)$. For the operator $P_1(-\frac{1}{\th})f=-\nc_1(-\frac{1}{\th})(\nu \wedge f)$ we have an approximation with  $T_{1, N}(h, z)f$ and by (\ref{eq:3.12}) the principal symbol of $T_{1, N}(h, -\ii)$ becomes
 $$m_1 = \frac{1}{\ii} \Bigl( \ii\sqrt{1 + r_0} I+ \frac{\bc}{\ii \sqrt{1 + r_0}}\Bigr) = \sqrt{1 + r_0} I - \frac{1}{\sqrt{1 + r_0}} \bc.$$

  \section{Properties of the operator $\pc(h)$}

In this section we study the case $\th$ real. Recall that  the operator $\bc(h D_{x'})$ has matrix symbol $\bc(x', h\xi')$ such that
$$\bc f = \la \beta, f \ra \beta, \: f \in \R^3,$$
where $\beta = \beta(x', \xi') \in \R^3$ is vector valued homogeneous polynomial of order 1 in $\xi'$ introduced in the previous section. The equality $\la \beta, \beta \ra = r_0$ implies
$\bc(x', \xi') \beta(x', \xi')  = r_0(x', \xi') \beta(x', \xi')$ and $\bc(x', \xi') (\nu(x')  \wedge \beta(x', \xi')) = 0$. Thus the matrix $\bc(x', \xi')$ has three eigenvectors $\nu(x') , \nu(x') \wedge \beta(x', \xi'), \: \beta(x', \xi')$ with corresponding eigenvalues $0, 0, r_0(x', \xi').$ These eigenvalues are defined globally on $\Gamma$. 
Let  $\|\xi'\|_g$ be the induced Riemann metric on $T^*(\Gamma)$ and 
let $b(x', \xi') = \beta(x', \frac{\xi'}{\sqrt{r_0(x', \xi')}})$. For $\|\xi'\|_g = 1$ introduce the unitary $(3 \times 3)$ matrix 
$$U(x', \xi') = \begin{pmatrix}\vert  & \vert & \vert \\ \nu(x') &  \nu(x') \wedge b(x', \xi') & b(x', \xi') \\ \vert & \vert & \vert \end{pmatrix} .$$
Then for $\|\xi'\|_g  = 1$ one obtains a global diagonalisation
$$ U^T (x', \xi') \bc(x', \xi') U(x', \xi')  = \begin{pmatrix} 0 & 0 & 0 \\0 & 0 & 0\\0 & 0 & r_0\end{pmatrix} ,$$
where $A^T$ denotes the transpose matrix of $A$. Writing $\xi'= \omega \|\xi'\|_g$ with $\|\omega\|_g = 1$ and using the fact that $B(x', \xi')$ and $r_0(x', \xi')$ are homogeneous of order 2 in $\xi'$, one concludes that  the above diagonalisation is true for all $\xi'$.

First we study the case $0 < \gamma(x) < 1,\: \forall x \in \Gamma,$ which yields $\gamma_0(x) =\frac{1}{\gamma(x)}.$ 
Introduce the self-adjoint operator $\pc(h)= -T_N(h, -\ii) - \gamma_0(x) I$ with principal symbol 
$$p_1 = \sqrt{1 + h^2 r_0} \Bigl (I - \frac{\bc(h \xi)}{1 + h^2 r_0}\Bigr) - \gamma_0(x) I= \tilde{p}_1 - \gamma_0(x)I ,$$
where  $I$ is the $(3 \times 3)$ identity matrix. We assume that $N$ is fixed sufficiently large and we omit this in the notation $\pc(h).$ Moreover, as it was mentioned in Section 2, we can write the pseudo-differential operator as a classical one and 
\begin{equation} \label{eq:4.1} 
 U^T p_1 U = \begin{pmatrix} \sqrt{1 + h^2 r_0} -\gamma_0& 0 & 0\\
0 & \sqrt{1 + h^2 r_0} -\gamma_0& 0\\
0 & 0 & (1 + h^2 r_0)^{-1/2} -\gamma_0,\end{pmatrix}.
\end{equation} 
Moreover, $U^{-1} = U^T$ and $U^T$ is the principal symbol of $(Op_h (U))^{-1}$. To examine the invertibility of $\pc(h)$, observe that the symbol 
$$- \gamma_0(x) \leq (1 + h^2 r_0)^{-1/2} - \gamma_0(x) \leq -(\gamma_0(x) - 1)$$
is elliptic.  Write $(Op_h(U))^{-1}  \pc(h)  Op_h(U)$ in a block matrix form
$$ (Op_h(U))^{-1}  \pc(h) Op_h(U) = \begin{pmatrix} R(h) & S(h) \\  S^*(h) &  r(h) \end{pmatrix}, $$
where $R(h)$ is a $(2 \times 2)$ matrix valued operator, $S(h)$ is $( 2 \times 1)$ matrix valued operator with symbol in $h \mathcal S^0_{0} (\Gamma),$ the adjoint operators $S^*(h)$ is $( 1 \times 2)$ matrix valued operator, while
$$\begin{pmatrix} R(h) & 0 \\ 0 & r(h)\end{pmatrix}$$
has principal symbol  (\ref{eq:4.1}). The equation $ (Op_h (U))^{-1}  \pc(h)Op_h(U)  (Y, y_3) = (F, f_3)$ with a vector $Y = (y_1, y_2)$ and $F = (f_1, f_2)$ implies 
$$r(h)  y_3 + S^*(h) Y = f_3.$$
Then  $y_3 = -r(h)^{-1} S^*(h) Y + r(h)^{-1} f_3$ and for $Y$ one obtains the equation
$$Q(h)Y = \Bigl( R(h)  -  S(h)r(h)^{-1} S^*(h) \Bigr) Y = F - S(h)  r(h)^{-1} f_3.$$
The invertibility of the operator 
$$Q(h): = R(h)  -  S(h)r(h)^{-1} S^*(h) $$  
depends of that of $R(h)$ and $R(h)$ 
has principal symbol 
$$q_1 = \begin{pmatrix} \sqrt{1 + h^2 r_0} - \gamma_0 & 0 \\ 0 & \sqrt{1 + h^2 r_0} - \gamma_0 \end{pmatrix}. $$

Let $$  c_0 = \min_{x \in \Gamma} \gamma_0(x) = (\max_{x \in \Gamma} \gamma(x))^{-1},\:   c_1 = \max_{x \in \Gamma} \gamma_0(x)= (\min_{x \in \Gamma} \gamma(x))^{-1} .$$ 
Introduce the constants $C = \frac{1}{c_1^2}, \: \ep = \frac{C}{2}(c_0 - 1)^2 < 1/2 $ and set $\la hD \ra =  (1 - h^2\Delta_{\Gamma})^{1/2}.$ We say that $A \geq B$ if $(A u, u) \geq (B u, u),\: \forall u \in L^2(\Gamma; \C^3).$ We need the following  

\begin{prop} The operator $Q(h)$ satisfies  the estimate
\begin{equation} \label{eq:4.2}
h \frac{\partial Q (h)}{\partial h} + C Q(h) \la h D \ra^{-1}  Q(h) \geq \ep \la h D \ra.
\end{equation} 
 \end{prop}
 
  \begin{proof} The proof is a repetition of that of Prop. 4.1 in \cite{P2}. For the sake of completeness we present the details. We have
$$h \frac{\partial q_1}{\partial h} = \frac{h^2 r_0}{\sqrt{ 1 + h^2 r_0}} I = \sqrt{1 + h^2 r_0}  I - (1 + h^2 r_0)^{-1/2} I ,$$
where $I$ is the $(2 \times 2)$ identity matrix.  The  operator $CQ(h) \la hD\ra^{-1} Q(h)$ has principal symbol
 $$C\sqrt{1 + h^2 r_0} I    - 2 C \gamma_0 I +C \gamma_0^2 (1 + h^2 r_0)^{-1/2} I$$
and the  principal symbol of the left hand side of (\ref{eq:4.2}) becomes 
$$( 1 + C - \ep) \sqrt{1 + h^2 r_0} I  + \ep \sqrt{1 + h^2 r_0} I  - 2 C \gamma_0I + (C \gamma_0^2 - 1) (1 + h^2 r_0)^{-1/2}I. $$
We write the last term in the form
$$(1 - C\gamma_0^2)\Bigl(1  - ( 1+h^2 r_0)^{-1/2} \Bigr) I  + (C\gamma_0^2 - 1) I = A_1 + A_2. $$
Since $ 1 - C \gamma_0^2(x) \geq 0$ and $1  - ( 1+h^2 r_0)^{-1/2} \geq 0$, the term $A_1$ is symmetric non-negative definite matrix and we may apply the semi-classical strict G\"arding inequality to bound from below $(Op_h(A_1)u, u)$ by $-C_1 h \|u\|^2.$ Next
$$(1 + C - \ep) ( \la h D \ra u, u) \geq (1 + C - \ep) \|u\|^2$$ 
and
$$\Bigl( (C(\gamma_0 - 1)^2- \ep)u, u\Bigr) \geq (C(c_0 - 1)^2 - \ep) \|u\|^2 =  \ep \|u \|^2.$$
The lower order symbol $h q_0$ of the operator $Q(h)$ yields a term 
$$h (Op(q_0)u, u) \geq - h\|Op(q_0)\|_{L^2 \to L^2} \|u\|^2 =- hC_2 \|u\|^2$$ 
and we may absorb these terms taking $0 < h \leq \ep(C_1 + C_2)^{-1} = \frac{\ep}{C_3}.$
\end{proof} 

For the analysis of the eigenvalues of $Q(h)$ we will follow the approach of \cite{SjV}. Introduce the semiclassical Sobolev space $H^s(\Gamma; \C^2)$ with norm $\|u\|_s = \|\la h D \ra^s u\|_{L^2(\Gamma; \C^2)}.$
Let 
$$\mu_1(h) \leq \mu_2(h) \leq ...\leq \mu_k(h) \leq ...$$
be the eigenvalues of $Q(h)$ repeated with their multiplicities. Fix $0 < h_0 \leq \frac{\ep}{C_3},$ where $\ep > 0$ is the constant in Proposition 4.1 and let $k _0 \in \N$ be chosen so that  $\mu_k(h_0) > 0 $ for $k \geq k_0.$ This follows from the fact that the number of the non-positive eigenvalues of $Q(h_0)$ is  given by a Weyl formula (see for instance Theorem 12.3 in \cite{DS})
$$(2\pi h_0)^{-2} \iint _{\sqrt{1 + r_0} - \gamma_0 \leq 0} dx'd\xi' + {\mathcal O}(h_0^{-1}).$$

 By using Proposition 4.1 and choosing $0 < \delta \leq \frac{c_0 - 1}{2}$, one obtains
$$\frac{\ep}{2} \leq h \frac{d \mu_k(h)}{dh} \leq C_0, \: k \geq k_0,$$
whenever $\mu_k(h) \in [-\delta, \delta], \: 0 < h \leq h_0$ (see Section 4 in \cite{P2}).   
Now  if $0 < \frac{1}{r} < h_0$ and $\mu_k(1/r) < 0$, then there exists unique $h_k, 1/r < h_k < h_0$ such that $\mu_k(h_k) = 0$. Clearly, the operator $Q(h_k)$ is not invertible and for the invertibility of $Q(h)$ we must avoid small intervals around $h_k$. The purpose is to obtain a bijection between the set of $h_k \in (0, h_0]$ and the eigenvalues in $\Lambda.$ 
Repeating the argument in Sections 4 in \cite{P2} and \cite{SjV},  one obtains the following
\begin{prop} [Prop.4.1, \cite{SjV}] Let $p > 3$ be fixed. The inverse operator $Q(h)^{-1} : L^2(\Gamma; \C^2) \rightarrow L^2(\Gamma; \C^2) $ exists and has norm ${\mathcal O}(h^{-p})$ for $h \in (0, h_0] \setminus \Omega_p,$ where $\Omega_p$ is a union of disjoint closed intervals $J_{1, p}, J_{2, p},...$ with $|J_{k,p}| = {\mathcal O} (h^{p -1})$ for $h \in J_{k,p}.$ Moreover, the number of such intervals that intersect $[h/2, h]$ for $0 < h \leq h_0$ is at most ${\mathcal O}(h^{1- p}).$ 
\end{prop} 

If the operator $Q(h)^{-1}$ exists, it is easy to see that $\pc(h)$ is also invertible. First, we have
\begin{equation} \label{eq:4.3} 
\begin{aligned}
 \begin{pmatrix} I &  S(h)r^{-1} (h) \\
0_{2,1}  & 1 \end{pmatrix}
\begin{pmatrix} Q(h) & 0_{1, 2} \\S^*(h) & r(h) \end{pmatrix}  = \begin{pmatrix} R(h) & S(h) \\  S^*(h) & r(h)\end{pmatrix},
\end{aligned} 
\end{equation}
where $I$ is the identity $(2 \times 2)$ matrix and $0_{1,2},\: 0_{2, 1}$ are $(1 \times 2)$ and $(2 \times 1)$ matrices, respectively,  with zero entires.
Second, the operator $r^{-1}(h)$ has principal symbol $\frac{\sqrt{1 + h^2 r_0}}{1 - \sqrt{1 + h^2 r_0} \gamma_0}\in \mathcal S^0_0,$ so $r^{-1}(h) : H^s(\Gamma; \C) \to H^s(\Gamma; \C)$ is bounded for every $s$. On the other hand,  $S(h): H^s(\Gamma; \C) \to H^s(\Gamma; \C^2)$ has norm ${\mathcal O}_s(h).$ Consequently, the  operator 
$$ \begin{pmatrix} I &  S(h)r^{-1} (h) \\
0_{2, 1}  &1 \end{pmatrix}^{-1}$$
is bounded in ${\mathcal L}(H^s(\Gamma; \C^3), H^s(\Gamma; \C^3))$, while
$$\begin{pmatrix} Q(h) & 0_{1, 2} \\S^*(h) & r(h) \end{pmatrix}^{-1} = \begin{pmatrix} Q(h)^{-1} & 0_{1, 2} \\ - r^{-1} (h) S^* (h) Q(h)^{-1} & r^{-1} (h) \end{pmatrix}. $$
We deduce that the operator on the right hand side of (\ref{eq:4.3}) is invertible, whenever $Q(h)$ is invertible and since $Op_h(U)$ is invertible this implies the invertibility of $\pc(h).$ Finally, the statement of Proposition 4.2 holds for the operator $\pc(h)^{-1}$ with the same intervals $J_{k, p}$ and we have a bound $\|\pc(h)^{-1}\|_{L^2(\Gamma; \C^3) \to L^2(\Gamma; \C^3)} = {\mathcal O} (h^{-p})$ for $h \in (0, h_0] \setminus \Omega_p.$

The analysis of the case $\gamma(x) > 1, \: \forall x \in \Gamma,$ is completely similar to that of the case $0 < \gamma(x) <1$ examined above and we have $\gamma_0(x) = \gamma(x).$ We study the operators $\nc_1(\lambda), C_1(\lambda)$
and $P_1(h) = P_1(- \frac{1}{h})$ introduced at the end of Section 2. For the self adjoint operator $\pc_1(h) = T_{1, N}(h, -\ii) - \gamma_0(x)I,$ the argument at the end of Section 3 shows that $\pc_1(h)$ has principal symbol $\tilde{p}_1(x', h \xi') - \gamma_0(x) I$.  Thus we obtain the statements of Proposition 4.1 and Proposition 4.2 with a self-adjoint operator $Q_1(h)$ having principal symbol
$$ \begin{pmatrix} \sqrt{1 + h^2 r_0} - \gamma & 0 \\ 0 & \sqrt{1 + h^2 r_0} - \gamma_0 \end{pmatrix}. $$ 
 Notice that both operators $Q(h), Q_1(h)$ have the same principal symbol. Next for the operator $\pc_1(h)$ we obtain the same statements as those for $\pc(h).$

\section{Relation between the trace integrals for $\pc(\th)$ and $\cc(\th)$}

The purpose in this section is to study the operators $\pc(\th)$ and $\cc(\th)$ for complex $\th = h( 1 + \ii t)\in L,\: |t|\leq h^2.$ We change the notations and we will use the notation $h$ for the points in $L \subset \C$ with $|\im h| \leq (\Re h)^2, \: 0 < \Re h \leq h_0 \ll 1.$ First we study the case $0 < \gamma(x) < 1, \: \forall x \in \Gamma.$ The operator $T_N(h, z)$ can be extended for $h \in L$ as a holomorphic function of $h$. The same is true for $\pc(h) = -T_N(h, z) - \gamma_0(x) I$. To study $\pc(h)^{-1}$, we must examine the inverse of the operator on the left hand side of (\ref{eq:4.3}) for $h \in L.$ Clearly, $S(h),\: S^*(h)$  and $r(h)$ can be extended for $h \in L$ and $h^{-1}S(h),\: r(h)^{-1}$ are bounded as operators from $H^s(\Gamma; \C)$ to $H^ s(\Gamma; \C^2)$ and from $H^s(\Gamma, \C)$ to $H^s(\Gamma, \C)$, respectively.  Since $b(x', h \xi') = b(x', \xi'),$  the symbol of $U(x', \xi') $ may be trivially extended for $h \in L.$ It remains to study $Q(h)^{-1}.$ Repeating the proof of Lemma 5.1 in \cite{SjV} and using Proposition 4.1, we get
\begin{equation} \label{eq:5.1}
\|Q(h)^{-1} \|_{\mathcal L (H^{-1/2}(\Gamma; \C^2), H^{1/2}(\Gamma; \C^2))} \leq C\frac{\Re h}{|\im h|},\: \im h \neq 0, \: h \in L.
\end{equation} 
Here we have used the estimate
$$\|r(h)^{-1} \|_{H^s(\Gamma; \C) \to H^s(\Gamma; \C)}  \leq C_s' \leq C_s'\frac{\Re h}{|\im h|},\: \im h \neq 0, \: h \in L$$
since $|\im\: h|  \leq (\Re  h)^2 \leq h_0 \Re  h .$
To obtain an estimate of 
$$\|Q(h)^{-1}\|_{{\mathcal L}(H^s(\Gamma; \C^2), H^{s+ 1}(\Gamma; \C^2))},$$ 
as in Section 5 in \cite{P2}, we introduce a $C^{\infty}$ symbol
$$ \chi(x', \xi') = \begin{cases} 2, \: x'\in \Gamma, \:\|\xi'\|_g \leq B_0,\\
0, \: x' \in \Gamma, \: \|\xi'\|_g \geq B_0 + 1.\end{cases}$$
Here $B_0 > 0$ is a constant such that $\sqrt{C_3} B_0 \geq 2 c_1,\: r_0(x', \xi') \geq C_3 \|\xi'\|_g^2.$ Then we extend homorphically $\chi(x', \Re h D_{x'})$ to $\zeta(x', h D_{x'})$ for $h \in L$ and consider the operator 
$M( h) = Q(h) + \gamma_0(x') \zeta(x', h D_{x'}).$ This modification implies the property $Q(h) - M(h):\: {\mathcal O}_s(1): H^{-s}(\Gamma; \C^2) \rightarrow H^s(\Gamma; \C^2)$ for every $s$ and the operator $M(h)$ with principal symbol $M(x', \xi') \in \mathcal S^1_{0}$ becomes elliptic. Then $M(h)^{-1} : H^s(\Gamma; \C^2) \rightarrow H^{s+1}(\Gamma; \C^2)$ is bounded by ${\mathcal O}_s(1)$  and repeating the argument in Section 5, \cite{P2} and using (\ref{eq:5.1}),   
one deduces
\begin{equation} \label{eq:5.2}
\|Q(h)^{-1}\|_{\mathcal L (H^s(\Gamma; \C^2), H^{s+ 1}(\Gamma; \C^2))} \leq C_s\frac{\Re h}{|\im h|},\: \im h \neq 0. 
\end{equation} 
 Taking the inverse operators in (\ref{eq:4.3}), one obtains with another constant $C_s$ the estimate
\begin{equation} \label{eq:5.3} 
\|\pc(h)^{-1}\|_{\mathcal L (H^s(\Gamma; \C^3), H^{s+ 1}(\Gamma; \C^3))} \leq C_s\frac{\Re h}{|\im h|},\: \im h \neq 0 .
\end{equation}
 
 Following \cite{SjV}, we introduce piecewise smooth positively oriented curve $\gamma_{k, p} \subset \C$ which is a union of four segments: $\Re h \in J_{k, p},\: \im h = \pm (\Re h)^{p+1}$ and $\Re h \in \partial J_{k, p},\: |\im h| \leq (\Re h)^{p+1}$, $J_{k, p}$ being the interval in $\Omega_p$ introduced in Proposition 4.3. 
  \begin{prop} For every $h \in \gamma_{k, p}$ the inverse operator $\pc(h)^{-1}$ exists and
 \begin{equation} \label{eq:5.4}
 \|\pc(h)^{-1} \|_{{\mathcal L} (H^s(\Gamma; \C^3) , H^{s+ 1} (\Gamma; \C^3))} \leq C_{k, s} (\Re h)^{-p}.
 \end{equation} 
 \end{prop} 
 The proof is the same as in Proposition 5.2 in \cite{SjV}. It is based on the estimate of $\|\pc(h)^{-1}\|_{L^2(\Gamma; \C^3) \to L^2(\Gamma; \C^3)}$ for $h \in (0, h_0] \setminus \Omega_p$, the Taylor expansion of $\pc(h)$ for $0 \leq |\im h| \leq (\Re h)^{p+1}$  and the application of (\ref{eq:5.3}). We omit the details.
 Of course, by the same argument an analog to (\ref{eq:5.4}) holds for the norm of the operator $Q(h)^{-1}$ and $h \in \gamma_{k, p}.$\\
 
 To obtain an estimate for $\cc(h)^{-1}$, with $N$ large enough write
 $$\cc(h)f = \nc(-\frac{1}{h})f +\gamma_0(x)  (\nu \wedge f)  = T_N(h, z)(\nu \wedge f) + \gamma_0(x) (\nu \wedge f) +   {\mathcal R}_q(h, z)(\nu \wedge f)$$
 $$ = - \pc(h)i_{\nu} f + {\mathcal R}_q(h, z)i_{\nu}f , \: q \gg 2p$$
with   $\mathcal R_q(h, z): {\mathcal O}_s((\Re h)^q)) : H^s \rightarrow H^{s + q - 1} .$  This yields
$$ \pc(h) ^{-1} \cc(h)f = -\Bigl(Id -\pc(h)^{-1} {\mathcal R}_q(h, z) \Bigr) i_{\nu}f$$
and by (\ref{eq:5.4}) one deduces
$$\big \| \pc(h)^{-1} {\mathcal R}_q( h, z) \big \| _{{\mathcal L}( H^s, H^{s + q})} \leq C_s (\Re h) ^{-p + q}.$$  
For small $\Re h$ this implies
$$i_{\nu} \Bigl(Id - \pc(h)^{-1} {\mathcal R}_q(h, z)\Bigr)^{-1} \pc(h)^{-1} \cc(h) = Id.$$   
Repeating the argument in Section 5 of \cite{P2}, we obtain
\begin{equation} \label{eq:5.5} 
\|\cc(h)^{-1} \|_{\mathcal L(H^s, H^{s+1})} \leq C_s (\Re h)^{-p},\: h \in \gamma_{k, p}.
\end{equation} 
In the same way writing
$$\cc(h)^{-1} - i_{\nu} \pc(h)^{-1} = i_{\nu} \Bigl( \Bigl(Id - \pc(h)^{-1} {\mathcal R}_q(h, z)\Bigr)^{-1} - Id\Bigr)  \pc(h)^{-1}, $$
one gets
\begin{equation} \label{eq:5.6} 
\| \cc(h)^{-1} - i_{\nu}\pc(h)^{-1} \|_{\mathcal L(H^s, H^{s + q -1})} \leq C_s (|h|^{q- 2p} ),\: h \in \gamma_{k, p}.
\end{equation} 
On the other hand, $i_{\nu} \pc(h)^{-1} = \Bigl(- \pc(h) i_{\nu}\Bigr)^{-1}$ since $i_{\nu} i_{\nu} = -Id.$
By using the Cauchy formula 
$$\frac{d}{dh} \Bigl(\cc(h) -  (-\pc(h)i_{\nu}\Bigr)= \frac{1}{2 \pi \ii} \int_{\tilde{\gamma}_{k, p}} \frac{C(\zeta) + \pc(\zeta) i_{\nu}}{\zeta - h} d\zeta$$
$$= \int_{\tilde{\gamma}_{k, p}} {\mathcal R}_q(\zeta, z) i_{\nu} d\zeta,$$
where $\tilde{\gamma}_{k, p}$ is the boundary of a domain containing $\gamma_{k, p}$, one deduces
\begin{equation}\label{eq:5.7} 
 \big\| \frac{d}{dh} {\cc}(h) - \frac{d}{dh} (-\pc(h)i_{\nu})\big\| _{{\mathcal L}(H^s, H^{q+ q- 1}) } \leq C_s (\Re h)^q.
\end{equation}
 
   Now we pass to a trace formula involving $\pc(h)^{-1}$ and $Q(h)^{-1}.$  Recall that $k_0 \in \N$ is fixed so that $\mu_k(h_0) > 0, \: k \geq k_0.$  Let $\mu_k(h_k) = 0, \: 0 < h_k < h_0,\: k \geq k_0.$ Since $\mu_k(h)$ is increasing when $\mu_k(h) \in [-\delta, \delta]$, the function $\mu_k(h)$ has no other zeros for $0 < h \leq h_0.$ We define the multiplicity of $h_k$ as the multiplicity of the eigenvalues $\mu_k(h)$ of $Q(h)$ and denote by $\dot{A}$ the derivative of $A$ with respect to $h$. 
      \begin{prop} Let $\beta \subset L$ be a closed positively oriented simple $C^1$ curve without self intersections such that there are no points $h_k$ on $\beta$ with $\mu_k(h_k) = 0, \: k \geq k_0$.Then
      \begin{equation} \label{eq:5.8}
       {\rm tr}_{H^{1/2}(\Gamma; \C^3)} \frac{1}{2 \pi \ii} \int_{\beta} \pc(h)^{-1}\dot{\pc}(h) dh =  {\rm tr}_{H^{1/2}(\Gamma; \C^2)} \frac{1}{2 \pi \ii} \int_{\beta} Q(h)^{-1}\dot{Q}(h) dh    
   \end{equation}  
   is equal to the number of $h_k$ counted with their multiplicities in the domain bounded by $\beta.$ 
      \end{prop} 
 \begin{proof} Since $\beta$ is related to the eigenvalues of $Q(h)$, repeating without any changes the argument of the proof of Proposition 5.3 in \cite{SjV}, one deduces the existence of the trace on the right hand side of (\ref{eq:5.8}) and the fact that this trace is equal to the number of $h_k$  in the domain bounded by $\beta.$ Next
 $$ \int_{\beta} \begin{pmatrix} Q(h) & 0_{1, 2} \\S^*(h) & r(h) \end{pmatrix}^{-1}  \begin{pmatrix} \dot{Q}(h) & 0_{1, 2} \\ \dot{S}^*(h) & \dot{r}(h) \end{pmatrix} dh= \int_{\beta} \begin{pmatrix} Q(h)^{-1} \dot{Q}(h) & 0 \\ Y_{1, 2} (h) & r^{-1} (h) \dot{r}(h) \end{pmatrix} dh$$
 and the integral of $r^{-1}(h) \dot{r}(h)$ vanishes since this operator is analytic in the domain bounded by $\beta.$ Thus the trace of the right  hand side of the above equality is equal to the right hand side of (\ref{eq:5.8}) multiplies by $(2 \pi \ii).$ 
 Write
 $$\begin{pmatrix} Q(h) & 0_{1, 2} \\S^*(h) & r(h) \end{pmatrix}^{-1}  \begin{pmatrix} I &  S(h)r^{-1} (h) \\
0_{2,1}  & 1 \end{pmatrix}^{-1} \frac{d}{dh} \Bigr[\begin{pmatrix} I &  S(h)r^{-1} (h) \\
0_{2,1}  & 1 \end{pmatrix} \begin{pmatrix} Q(h) & 0_{1, 2} \\S^*(h) & r(h) \end{pmatrix}\Bigr]$$
$$= \begin{pmatrix} Q(h) & 0_{1, 2} \\S^*(h) & r(h) \end{pmatrix}^{-1} \begin{pmatrix} \dot{Q}(h) & 0_{1, 2} \\\dot{S}^*(h) & \dot{r}(h) \end{pmatrix}+ Z(h).$$
The  integral of $Z(h)$ vanishes  by the cyclicity of trace since the product 
$$\begin{pmatrix} I & S^*(h)r^{-1}(h) \\ 0_{2,1}  & 1 \end{pmatrix}^{-1} \frac{d}{dh} \begin{pmatrix} I &  {S}^*(h)r^{-1} (h) \\
0_{2,1}  & 1 \end{pmatrix}.$$
 is an analytic function of $h$.  By applying the equality (\ref{eq:4.3}), we obtain that the trace of integral involving $\begin{pmatrix} R(h) & S(h) \\  S^*(h) & r(h)\end{pmatrix}$ is equal to the trace on the right hand side of (\ref{eq:5.8}). By the same manipulation as above taking the product with $(Op_h(U))^{-1}$ on the right and by $Op_h(U)$ on the left, one obtains (\ref{eq:5.8}).  
\end{proof} 
 
Notice that by the cyclicity of the trace we get
 $${\rm tr}_{H^{1/2}(\Gamma; \C^3)}  \int_{\beta} \pc(h)^{-1}\dot{\pc}(h) dh = {\rm tr}_{H^{1/2}(\Gamma; \C^3)}  \int_{\beta}\Bigl(- \pc(h)i_{\nu}\Bigr)^{-1}\frac{d}{dh} \Bigl(- \pc(h)i_{\nu} \Bigr) dh.$$ 
  Applying the estimate (\ref{eq:5.6})  for $\cc(h)^{-1} - i_{\nu} \pc(h)^{-1} $ and (\ref{eq:5.7})  for $\frac{d}{dh}\cc(h) - \frac{d}{dh}\Bigl(- \pc(h)i_{\nu}\Bigr)$ and taking into account Proposition 5.2, we conclude as in Section 5 of \cite{P2} that in the case $0 < \gamma(x) < 1, \: \forall x \in \Gamma,$ we have
 $${\rm tr}\: \frac{1}{2 \pi \ii} \int_{\gamma_{k, p}} \cc(h)^{-1} \dot{\cc}(h) dh = {\rm tr}\: \frac{1}{2 \pi \ii} \int_{\gamma_{k, p}} \pc(h)^{-1} \dot{\pc}(h) dh $$
 $$= {\rm tr}\: \frac{1}{2 \pi \ii} \int_{\gamma_{k, p}} Q(h)^{-1} \dot{Q}(h) dh.$$
 The analysis in the case $\gamma(x) > 1, \: \forall x \in \Gamma,$ is completely similar and we have trace formula involving the operator $\cc_1(h)$ introduced at the end of  Section 2 and trace formula involving $Q_1(h)$ and $\pc_1(h) = T_{1, N}(h, z) - \gamma_0 I.$ In this case 
 $${\rm tr}\: \frac{1}{2 \pi \ii} \int_{\gamma_{k, p}} \cc_1(h)^{-1} \dot{\cc}_1(h) dh = {\rm tr}\: \frac{1}{2 \pi \ii} \int_{\gamma_{k, p}} \pc_1(h)^{-1} \dot{\pc}_1(h) dh $$
 $$= {\rm tr}\: \frac{1}{2 \pi \ii} \int_{\gamma_{k, p}} Q_1(h)^{-1} \dot{Q}_1(h) dh.$$ 
 
 The equality of traces shows that the proof of the asymptotic (\ref{eq:1.3}) is reduced to the count of $h_k$ with their multiplicities for which we have $\mu_k(h_k) = 0$ in the domain $\beta_{k, j}$ bounded by $\gamma_{k, p}$. Here $\mu_k(h)$ are the eigenvalues of $Q(h)$ (resp. $Q_1(h)$) if $0 < \gamma(x) < 1$ (resp. if $\gamma(x) > 1$).  We obtain  a bijection $\beta_{k, j} \ni h_k \rightarrow \ell(h_k) = \lambda_j \in \sigma_p(G_b) \cap \Lambda$ which preserves the multiplicities. The existence of  $h_k$ with $1/r < h_k< h_0$ is equivalent to $\mu_k(1/r) < 0$ and we are going to study the asymptotic of the counting function of the negative eigenvalues of $Q(1/r)$ (resp. $Q_1(1/r)$). The semiclassical principal symbol of both operators $Q(h),\: Q_1(h)$ has a double eigenvalue $q(x', \xi') = \sqrt{1 + r_0(x', \xi')} - \gamma_0(x')$. Applying Theorem  12.3 in \cite{DS}, we obtain
 $$\sharp\{ \lambda \in \sigma_p(G_b) \cap \Lambda: \: |\lambda| \leq r, \: r \geq C_{\gamma_0} \} = \frac{ r^2}{(2 \pi)^2} \int_{q(x', \xi') \leq 0} dx' d \xi' + {\mathcal O}_{\gamma_0}(r).$$
 Finally,
 $$\int_{q(x', \xi') \leq 0} dx'd\xi' = \int_{ r_0(x', \xi') \leq \gamma_0^2(x') - 1} dx'd\xi' = \pi \int_{\Gamma} (\gamma_0^2(x') - 1) dx'$$
and this completes the proof of Theorem 1.1.

\end{document}